\documentclass[10pt,reqno]{amsart}
\usepackage{amssymb,amsmath,amscd,amsfonts,amsthm}
\usepackage[cp866]{inputenc}
\usepackage[english]{babel}

\newcommand{\Ker}{\mathrm{Ker\,}}
\newcommand{\Coker}{\mathrm{Coker\,}}
\newcommand{\coker}{\mathrm{coker\,}}
\newcommand{\im}{\mathrm{im\,}}
\newcommand{\coim}{\mathrm{coim\,}}

\newcommand{\id}{\mathrm{id}}

\newtheorem{theorem}{Theorem}[section]
\newtheorem{lemma}{Lemma}[section]
\newtheorem{corollary}{Corollary}[section]

\begin{document}

\title[The Two-Square Lemma and the Connecting Morphism]
{The Two-Square Lemma and the Connecting Morphism}
\thanks{
Partially supported by the Russian Foundation for Basic Research (Grant~09-01-00142-a),
the State Maintenance Program for the Leading Scientific Schools and
Junior Scientists of the Russian Federation (NSh-6613.2010.1),
and the Integration Project ``Quasiconformal Analysis and Geometric Aspects 
of Operator Theory'' of the Siberian Branch
and the Far East Branch of the Russian Academy of Sciences}
\author{Yaroslav Kopylov}
\address{Yaroslav Kopylov
\newline\hphantom{iii} Sobolev Institute of Mathematics,
\newline\hphantom{iii} Pr. Akademika Koptyuga 4,
\newline\hphantom{iii} 630090, Novosibirsk, Russia
\vspace{3mm}
\newline\hphantom{iii} Novosibirsk State University}%
\email{yakop@math.nsc.ru}
\maketitle
\begin{abstract}
We obtain a generalization of the Two-Square Lemma proved for abelian categories
by Fay, Hardie, and Hilton in 1989 and (in a special case) for preabelian categories
by Generalov in 1994.  We also prove the equivalence up to sign of two definitions of a connecting
morphism of the Snake Lemma.
\end{abstract}

\section*{Introduction}

One of the most important diagram assertions in homological algebra is the so-called Snake Lemma
which makes it possible to obtain homological sequences from short exact sequences of complexes.
It always holds in an abelian category. However, in the more general context of preabelian categories,
The Snake Lemma fails without additional assumptions on the initial diagram. The main reasons are
that the notions of kernel and monomorphism (respectively, of cokernel and epimorphism) do not coincide
in a preabelian category and that kernels (respectively, cokernels) do not ``survive'' under
pushouts (respectively, pullbacks).

The validity of the Snake Lemma in the nonabelian case was studied by several authors
for classes of additive categories (see, e.g., \cite{Bue2010,Ge94,GlK02,KoK00,KoK09}) and in some classes 
of nonabelian categories (see, e.g., \cite{BoBou04, Grand91-1}).
The key properties of the morphisms in the initial diagram required for the exactness of
the $\Ker$-$\Coker$-sequence are ``strictness'' and stability under pushouts (pullbacks) of some
monomorphisms (epimorphisms), or their weaker analogs ``exactness'' and ``modularity'' \cite{Grand91-1}.

Even the existence of a connecting morphism, valid in abelian categories (and even in quasi-abelian
categories \cite{KoK00} and in their nonadditive counterpart, Grandis homological categories \cite{Grand91-1}),
cannot be guaranteed in general preabelian categories without extra ``semi-stability'' assumptions
(see \cite{Ge94}). The construction of the connecting morphism in \cite{Ge94} involves a preabelian
version of a special case of the Two-Square Lemma of Fay--Hardie-Hilton \cite[Lemma~3]{FHaHi89}.

\begin{theorem}[The Two-Square Lemma]\label{l2sq}
Suppose that the following diagram in an abelian category has exact rows:
\begin{equation}\label{main}
\CD
A @>\psi>> B @>\varphi>> C  \\
@V\alpha VV @V\beta VV @V\gamma VV  \\
A' @>>\psi'> B' @>>\varphi'> C'\,.
\endCD
\end{equation}
Let
\begin{equation}\label{pb1}
\CD
Q' @>\sigma>> C \\
@V\sigma' VV  @V\gamma VV \\
B' @>>\varphi'> C'
\endCD
\end{equation}
be a pullback and let
\begin{equation}\label{po1}
\CD
A @>\psi>> B \\
@V\alpha VV  @V\tau VV \\
A' @>>\tau'> Q
\endCD
\end{equation}
be a pushout.

Then

(i) there exists a unique $\theta:Q\rightarrow B'$ such that $\theta\tau=\beta$, $\theta\tau'=\psi'$;

(ii) there exists a unique $\rho:B\rightarrow Q'$ such that $\sigma\rho=\varphi$, $\sigma'\rho=\beta$;

(iii) there exists a unique $\eta:Q\rightarrow Q'$ such that $\eta\tau=\rho$, $\sigma'\eta=\theta$,
$\sigma\eta\tau'=0$.
\end{theorem}

The proof in \cite{FHaHi89} remains valid in any preabelian category. The Two-Square Lemma of
\cite{FHaHi89} also claims that \textit{if $\psi'$ is a monomorphism then so is $\eta$} and 
\textit{if $\varphi$ is an epimorphism then so is $\eta$}. 

In~\cite{Ge94}, Generalov proved the following assertion:

\begin{theorem}\label{geniso}
Consider a diagram of the form~(\ref{main}) in a preabelian category. If $\psi'$ is a semi-stable
kernel and $\varphi$ is a semi-stable cokernel then $\eta$ is an isomorphism.
\end{theorem}

Below we study the question when $\eta$ is monic, epic, a kernel, a cokernel in a preabelian category.

The article is organized as follows. In Section~\ref{preabcat}, we give basic definitions and
facts about preabelian categories. In Section~\ref{twos}, we prove the main assertion
of the article, Theorem~\ref{twosql}, explaining what conditions on the initial diagram~(\ref{main})
guarantee each of the above-mentioned properties of $\eta$. In Section~3, we prove the equivalence
of two definitions of a connecting morphism of the Snake Lemma in a preabelian category.

\section{Preabelian Categories}\label{preabcat}

A {\it preabelian category} is an additive category with kernels and cokernels.

In a preabelian category, every morphism $\alpha$ admits a canonical decomposition
$$
\alpha=(\im\alpha)\bar\alpha(\coim\alpha),
\quad
\text{where }\im\alpha=\ker\coker\alpha,
\
\coim\alpha=\coker\ker\alpha.
$$
A morphism $\alpha$ is called {\it strict\/}
if $\bar\alpha$ is an isomorphism. A preabelian category is abelian if and only if
every morphism in it is strict. Note that
\begin{gather*}
\text{strict monomorphisms $=$ kernels}, \\
\text{strict epimorphisms $=$ cokernels}.
\end{gather*}

\begin{lemma}\label{preab} \cite{BD68,Ke69,KoK00,RiWa77}
The following hold in a preabelian category.

(i)~A morphism $\alpha$ is a kernel if and only if $\alpha=\operatorname{im}\alpha$,
a morphism $\alpha$ is a cokernel if and only if $\alpha=\coim\alpha$;

(ii)~A morphism $\alpha$ is strict
if and only if $\alpha$  is representable as
$\alpha=\alpha_1 \alpha_0$, where
$\alpha_0$ is a cokernel, $\alpha_1$ is a kernel; in~ this case,
$\alpha_0=\coim\alpha$ and $\alpha_1=\operatorname{im}\alpha$;

(iii)~Suppose that the commutative square
\begin{equation}\label{mainsq}
\CD
C @>\alpha>> D \\
@VgVV  @VVfV \\
A @>>\beta> B
\endCD
\end{equation}
is a pullback. Then $\ker f=\alpha\ker g$. If $f=\ker h$ for some $h$ then $g=\ker(h\beta)$.
In particular, if $f$ is monic then $g$ is monic; if $f$ is a kernel then
$g$ is a kernel.

In the dual manner, assume that (\ref{mainsq}) is a pushout. Then $\coker f=\beta\coker g$.
If $g=\coker e$ for some $e$ then $f=\coker(\alpha e)$. In particular, if $g$ is epic
then $f$ is epic; if $g$ is a cokernel then $f$ is a cokernel.
\end{lemma}

A kernel $g$ in a preabelian category is called {\it semi-stable} \cite{RiWa77} if for every 
pushout of the form~(\ref{mainsq}) $f$ is a kernel too.
A {\it semi-stable cokernel} is defined in the dual way. Examples of non-semi-stable cokernels may be found,
for example, in \cite{BoDi06,Pr00,Ru08,Ru11} and non-semi-stable kernels are shown in~\cite{RiWa77}.
If all kernels and cokernels are semi-stable then the preabelian category is called {\it quasi-abelian}
\cite{Sch99}. 

\begin{lemma}\label{semistab}\cite{Ge92,RiWa77}
The following hold in a preabelian category:

(i) if $gf$ is a semi-stable kernel then so is $f$; if $gf$ is a semi-stable cokernel then so is $g$;

(ii) if $f$ and $g$ are semi-stable kernels (cokernels) and the composition $gf$ is defined then
$gf$ is a semi-stable kernel (cokernel);

(iii) a pushout of a semi-stable kernel is a semi-stable kernel; a pullback of a semi-stable cokernel
is a semi-stable cokernel.
\end{lemma}

If the category satisfies the following two weaker axioms dual to one another then it is called
{\it P-semi-abelian} or {semi-abelian in the sense of Palamodov} \cite{Pa71}:  if (\ref{mainsq}) is a pushout and
$g$ is a kernel then $f$ is monic; if (\ref{mainsq}) is a pullback and $f$ is a cokernel then $g$ is
epic. Until recently it was unclear whether every P-semi-abelian category is quasi-abelian (Raikov's
Conjecture); this was disproved by Bonet and Dierolf \cite{BoDi06} and Rump \cite{Ru08,Ru11}. It turned out that,
for instance, the categories of barrelled and bornological spaces are P-semi-abelian but not quasi-abelian
(see \cite{Ru11}). In general preabelian categories, kernels (cokernels) may even push out (pull back)
to zero morphisms (see \cite{Pr00,RiWa77}).

In \cite{KCh72} Kuz$'$minov and Cherevikin proved that a preabelian category
is P-semi-abelian in the above sense if and only if, in the canonical decomposition
of every morphism $\alpha$, $\alpha=(\im\alpha)\bar\alpha\,\coim\alpha$, the central
morphism $\bar\alpha$ is a bimorphism, that is, monic and epic simultaneously.

\begin{lemma}\label{psemiab} \cite{KoK09,KCh72}
The following hold in a P-semi-abelian category{\rm:}

(i) if $gf$ is a kernel then $f$ is a kernel;
if $gf$ is a cokernel then $g$ is a cokernel;

(ii) if $f,g$ are kernels and $gf$ is defined then $gf$
is a kernel; if $f,g$ are cokernels and $gf$ is defined then $gf$
is a cokernel;

(iii) if $gf$ is strict and $g$ is monic then $f$ is strict; if
$gf$ is strict and $g\in P$ then $f$ is strict.
\end{lemma}

The following lemma is due to Yakovlev \cite{Yak79}.

\begin{lemma}\label{yak}
For every morphism $\alpha$ in a preabelian category,
$\ker\alpha=\ker\coim\alpha$, $\coker\alpha=\coker\im\alpha$.
\end{lemma}

A sequence $\dots\overset{a}\rightarrow B \overset{b}\rightarrow \dots$ in a preabelian
category is said to be {\it exact at\/} $B$ if $\operatorname{im} a= \ker b$.
As follows from Lemma~\ref{yak}, this is equivalent to the fact that $\coker a = \coim b$.

\section{The Two-Square Lemma}\label{twos}

We begin with a lemma which, being itself of an independent interest, will be used
below. It is a generalization of \cite[Theorem~3]{KCh72} and \cite[Lemma~6]{Ko09}.

\begin{lemma}\label{rumb}
Let
$$
\begin{CD}
A @>p_1>> B_1 @>q_1>> C\\
@| @Vr VV @| \\
A @>>p_2> B_2 @>>q_2> C
\end{CD}
$$
be a commutative diagram in a preabelian category.

(i) If $p_1=\ker q_1$, $q_2 p_2=0$, $p_2$ is monic then $r$ is monic.

(ii) Suppose that $p_1=\ker q_1$, $p_2=\ker q_2$, $p_2$ and $\im q_1$
are semi-stable kernels, and $q_1$ is strict. Then $r$ is a semi-stable kernel.

The dual assertions also hold.
\end{lemma}

\begin{proof} (i) Assume that $rx=0$ and prove that then $x=0$. We have
$q_1 x= q_2 rx=0$. Since $p_1=\ker q_1$, this gives $x=p_1 y$ for some
$y$. Then $p_2 y= r p_1 y= rx=0$. Since $p_2$ is monic, $y$=0 and thus
$x=p_1 y=0$.

(ii) Decompose $q_1$ as $q_1=q'_1 q''_1$, $q_2=q'_2 q''_2$, where
$q''_j=\coim q_j:K_j\rightarrow C$, $j=1,2$. By assumption, $q'_1=\im q_1$.
Since $\coim q_1=\coker p_1$ and $\coim q_2=\coker p_2$, there is a
unique morphism $w:K_1\rightarrow K_2$ such that $w\,\coim q_1=(\coim q_2) r$.
For this $w$, $q'_1=q'_2 w$. Since $q'_1$ is a semi-stable kernel by hypothesis,
so is $w$ (Lemma~\ref{semistab}).

Consider the pushout
\begin{equation}
\CD
A @>p_1>> B_1 \\
@V p_2 VV  @V u_2 VV \\
B_2 @>u_1>> F.
\endCD
\end{equation}
Since $u_1 r p_1=u_1 p_2=u_2 p_1$, we have $(u_1 r - u_2) p_1=0$. Therefore, there exists
a unique morphism $s:K_1\rightarrow F$ with the property $u_1 r - u_2 = s\,\coim q_1$.

Consider the pushout
\begin{equation}
\CD
K_1 @>w>> K_2 \\
@V s VV  @V s' VV \\
F @>w'>> S.
\endCD
\end{equation}
Put $\mu=w'u_1-s'\coim q_2$. We infer
\begin{multline*}
\mu r=(w'u_1-s'\coim q_2)r = w'u_2 + w's\,\coim q_1 - s'(\coim q_2)r \\
=w'u_2  + w's\,\coim q_1 - w's\,\coim q_1 = w'u_2.
\end{multline*}
Thus, $\mu r=w'u_2$. Since $p_2$ and $w$ are semi-stable kernels, so are
$u_2$ and $w'$. Now, by Lemma~\ref{semistab}(ii), $\mu r=w'u_2$ is a semi-stable kernel 
as a composition of semi-stable kernels. Thus, by Lemma~\ref{semistab}(i), $r$ is 
a semi-stable kernel. The lemma is proved.
\end{proof}

We will also need the following preabelian version of Lemma~1 in \cite{FHaHi89}. This also
generalizes Lemma~\ref{preab}(iii).

\begin{lemma}\label{exrows}
If in the commutative diagram
\begin{equation}\label{genker}
\CD
@. B @>\varphi>> C  \\
@. @V\beta VV @V\gamma VV  \\
A' @>>\psi'> B' @>>\varphi'> C'
\endCD
\end{equation}
the square $\varphi'\beta=\gamma\varphi$ is a pullback and the bottom row of~(\ref{genker})
is exact then there is a unique morphism $\psi:A'\rightarrow B$ such that $\beta\psi=\psi'$,
$\varphi\psi=0$. If, in addition, $\bar\psi'$ is epic then the sequence
\begin{equation}\label{seqtop}
A'\overset{\psi}\longrightarrow B \overset{\varphi}\longrightarrow C
\end{equation}
is exact.

The dual assertion about pushouts also holds.
\end{lemma}

\begin{proof}  The existence and uniqueness follows from the equalities
$\varphi'\psi'=\gamma 0$. Now, suppose that $\bar\psi'$ is epic. Then, by Lemma~\ref{preab}(iii),
$\beta\ker\varphi=\ker\varphi'=\im\psi'$. Put $\psi=(\ker\varphi)\bar\psi'\coim\psi'$.
Then $\coker\psi=\coker\ker\varphi=\coim\varphi$, which is the exactness of
sequence (\ref{seqtop}).
\end{proof}

\begin{theorem}\label{twosql}
Consider a commutative diagram with exact rows of the kind (\ref{main}) in
a preabelian category. Preserve all notations of Theorem~\ref{l2sq}. Then the following hold:

(i) If $\psi'$ is a semi-stable kernel and, in the canonical decomposition of $\varphi$,
$\varphi=(\im\varphi)\bar\varphi\,\coim\varphi$, $\bar\varphi$ is a monomorphism then
$\eta$ is a monomorphism.

If $\varphi$ is a semi-stable cokernel and, in the canonical
decomposition of $\psi'$, $\psi'=(\im\psi')\bar\psi'\coim\psi'$, $\bar\psi'$
is an epimorphism then $\eta$ is an epimorphism.

(ii) If $\psi'$ and $\im\varphi$ are semi-stable kernels and $\varphi$ is strict then
$\eta$ is a semi-stable kernel.

If $\varphi$ and $\coim\psi'$ are semi-stable cokernels and $\psi'$ is
strict then $\eta$ is a semi-stable cokernel.
\end{theorem}

\begin{proof} (i) Since $\psi'=\sigma'\tau'$ is a semi-stable kernel, so is 
$\tau'$ (Lemma~\ref{semistab}). 

The commutative diagram
$$
\begin{CD}
A @>\psi>> Q' @>\varphi>> C\\
@V\alpha VV @VV\tau V  @| \\
A' @>>\tau'> Q @>>\sigma\eta> C
\end{CD}
$$
has a pushout on the left and an exact top row, and is such that $\bar\varphi$
is monic. By Lemma~\ref{preab}(iii), we infer that $\tau'=\ker(\sigma\eta)$.
Assume now that $\eta z=0$ for some $z:Z\rightarrow Q$.
We have $\sigma\eta z=0$, and hence $z=\tau' z'$ for some $z'$. We infer
$$
\psi' z'= \theta\tau' z'=\theta z=\sigma'\eta z=0.
$$
Since $\psi'$ is a monomorphism, $z=0$. Thus, $\eta$ is a monomorphism.

The second assertion of~(i) is dual to the first.

(ii) We have already noticed that $\tau'=\ker(\sigma\eta)$. Note also that $\eta\tau'=\ker\sigma$.
Indeed, we have the commutative diagram
$$
\begin{CD}
A' @>\eta\tau'>> Q' @>\sigma>> C\\
@| @V\sigma' VV @VV\gamma V \\
A' @>>\psi'> B' @>>\varphi> C'
\end{CD}
$$
in which $\psi'=\ker\varphi'$ and the square on the right is a pullback.
Hence, $\eta\tau'=\ker\sigma$.

Since $\tau\psi=\tau'\alpha$ is a pushout, we have $(\coker\tau')\tau=\coker\varphi=\coim\psi$.
Hence,
$$
(\im\varphi)(\coker\tau')\tau=(\im\varphi)\coim\varphi=\varphi=\sigma\eta\tau.
$$
Moreover, $(\im\varphi)(\coker\tau')\tau'=0$ and $\sigma\eta\tau'=0$. We infer that
$$
(\sigma\eta-(\im\varphi)\coker\tau')\tau=0, \quad (\sigma\eta-(\im\varphi)\coker\tau')\tau'=0.
$$
Since the zero morphism $0:Q\rightarrow C$ is the only morphism $y$ for which $y\tau=0$
and $y\tau'=0$, we infer that $(\im\varphi)\coker\tau'-\sigma\eta=0$. Therefore,
the morphism $\sigma\eta=(\im\varphi)\coker\tau'$ is strict.

Now, we arrive at the commutative diagram
$$
\begin{CD}
A' @>\tau'>> Q @>\sigma\eta>> C\\
@| @V\eta VV @| \\
A' @>>\eta\tau'> Q' @>>\sigma> C,
\end{CD}
$$
where $\tau'=\ker(\sigma\eta)$, $\eta\tau'=\ker\sigma$, $\eta\tau'$ a semi-stable
kernel (because $\psi'=\sigma'\eta\tau'$ is a semi-stable kernel), $\sigma\eta$
is strict, and $\im(\sigma\eta)=\im\varphi$ is a semi-stable kernel. By Lemma~\ref{rumb},
we see that $\eta$ is a semi-stable kernel.

The first assertion of ~(ii) is proved, and the second is dual to the first.
\end{proof}

Observe that the only thing we really need from the semi-stability of $\psi'$ (or $\varphi$)
in the proof of Theorem~\ref{twosql}(i) is the implication 
$$
\psi'~\mbox{is a kernel} 
\Longrightarrow \tau'~\mbox{is a kernel} \quad (\varphi~\mbox{is a cokernel} 
\Longrightarrow \sigma~\mbox{is a cokernel}).
$$
By Lemma~\ref{psemiab}(i), this assertion holds for arbitrary kernels (cokernels)
in a P-semi-abelian category. Thus, we have

\begin{corollary}\label{isom}
Consider a commutative diagram with exact rows of the kind (1) in a P-semi-abelian category.
Then the following hold.

(i) If $\psi'$ is a kernel then $\eta$ is a monomorphism. If $\varphi$ is a cokernel 
then $\eta$ is an epimorphism.

(ii) If $\psi'$ is a semi-stable kernel and $\varphi$ is a cokernel (or if $\psi'$ is a kernel and 
$\varphi$ is a semi-stable cokernel) then $\eta$ is an isomorphism.

\end{corollary}

\section{Two Definitions of a Connected Morphism}\label{conn}
Consider the commutative diagram
\begin{equation}\label{snake}
\CD
@. A @>\psi>> B @>\varphi>> C @>>> 0 \\
@. @V\alpha VV @V\beta VV @V\gamma VV @. \\
0 @>>> A' @>>\psi'> B' @>>\varphi'> C' @.
\endCD
\end{equation}
with $\psi'=\ker\varphi'$ and $\varphi=\coker\psi$
in a preabelian category.

As in the abelian  case, (\ref{snake}) gives rise
to two parts of a $\Ker$-$\Coker$-sequence (the composition of two consecutive arrows is zero):
$$
\Ker\alpha \overset\varepsilon\rightarrow \Ker\beta \overset\zeta\rightarrow
\Ker\gamma
$$
and
$$
\Coker\alpha \overset\tau\rightarrow
\Coker\beta \overset\theta\rightarrow\Coker\gamma.
$$

In contrast to the case of an abelian category (or even
a Grandis-homological \cite{Grand91-1} or a quasi-abelian \cite{KoK00}) category,
for preabelian categories, it is in general impossible to construct
a natural connecting morphism $\delta:\Ker\gamma\rightarrow\Coker\alpha$.
We will duscuss two constructions of $\delta$, one going back
to Andr\'e--MacLane, and the other based on the Two-Square Lemma,
which was proposed by Fay--Hardie--Hilton in \cite{FHaHi89}
for abelian categories and adapted to the preabelian case by Generalov
in \cite{Ge94}.

\subsection{The Andr\'e--MacLane construction}

According to \cite{Be79}, the following construction, described in \cite[p.~203]{MacL72}
for abelian categories, is due to Andr\'e--MacLane. It was used 
in \cite{KoK00,KoK09} for quasi-abelian and $P$-semi-abelian categories.

Let
\begin{equation}\label{snpb}
\CD
X @>s>> \Ker\gamma \\
@VuVV  @VV\ker\gamma V \\
B @>>\varphi> C
\endCD
\end{equation}
be a pullback and let
\begin{equation}\label{snpo}
\CD
A' @>\psi'>> B' \\
@V\coker\alpha VV  @VVvV \\
\Coker\alpha @>>t> Y
\endCD
\end{equation}
be a pushout.

Instead of semi-stability properties of universal nature,
impose on our situation appropriate ad hoc ``modularity''
conditions a la Grandis \cite{Grand91-1}:

\textbf {Assumptions A.}\,\,{\textit In (\ref{snpb})
$s$ is epic and in (\ref{snpo}) $t$ is a kernel.}

Assumptions~A are fulfilled in a preabelian category if
$\psi'$ is a semi-stable kernel and $\varphi$ is a semi-stable cokernel.
In a P-semi-abelian category, the semi-stability of $\psi'$
is already enough.

Since (\ref{snpo}) is a pushout,
$(\coker t)v=\coker\psi'=\coim\varphi'$. If we put $(\im\varphi')\bar\varphi'=\chi$
then $\varphi'=\chi(\coker t)v$. We have
$$
v\beta\psi=v\psi'\alpha=t(\coker\alpha)\alpha=0.
$$
Therefore, $v\beta=n\varphi$ for some unique $n$. In the dual manner,
$\psi'\beta u=0$, and hence $\beta u=\psi'm$ for a unique morphism $m$.
We infer
$$
(\coker t)n(\ker\gamma)s= (\coker t)n\varphi u = (\coker t)v\beta u
=(\coker\psi')\psi' m=0.
$$
Since $s$ is epic, this implies that $(\coker t)n\ker\gamma=0$. Since
$t=\ker\coker t$, we conclude that $n\ker\gamma=t\delta_I$ for some unique
$\delta_I$. This morphism $\delta_I$ is characterized uniquely by the property
\begin{equation}\label{defdel}
t\delta_I s=v\beta u.
\end{equation}

By duality, consider

\textbf{Assumptions A$^*$.}\,\,{\textit In (\ref{snpb})
$s$ is a cokernel and in (\ref{snpo}) $t$ is monic.}

In this case, we also obtain a morphism $\delta_I$ characterized by
(\ref{defdel}). Therefore, the two morphisms coincide provided that $s$ is a cokernel
and $t$ is a kernel simultaneously.

\subsection{The Fay--Hardie--Hilton--Generalov construction}
Consider diagram~(\ref{snake}) and suppose the fulfillment of one of the following 
conditions (i) and (ii).

(i) The ambient category is preabelian, $\psi'$ is a semi-stable kernel, and 
$\varphi$ is a semi-stable cokernel;

(ii) The ambient category is P-semi-abelian and $\psi'$ is a semi-stable kernel {\em or} 
$\varphi$ is a semi-stable cokernel.
 
Below we use all notations of the previous subsection and Section~\ref{twos}.

Generalov's Theorem (Theorem~\ref{geniso}) or Theorem~\ref{twosql} for~(i) 
and Corollary~\ref{isom} for~(ii) imply that 
in these cases the morphism $\eta:Q\rightarrow Q'$ of \cite{FHaHi89} is an isomorphism,
and so we may assume that $Q=Q'$, $\eta=\id_Q$.
Since (\ref{po1}) is a pushout, $\coker\tau=(\coker\alpha)\tau'$;
by duality, since (\ref{pb1}) is a pullback,
$\ker\sigma'=\sigma(\ker\gamma)$. Put
$$
\delta_{II}=(\coker\tau)\ker\sigma'.
$$

\begin{theorem}
The equality $\delta_{II}=-\delta_I$ holds.
\end{theorem}

\begin{proof}
Prove that the morphism $-\delta_{II}$ meets (\ref{defdel}), that is,
that $t\delta_{II}s=-v\beta u$. 

We have
\begin{gather*}
(v\sigma'-n\sigma)\tau'=v\sigma'\tau'-n\sigma\tau'=v\psi'=t\coker\alpha=t \delta_1\tau', \\
(v\sigma'-n\sigma)\tau=v\beta-n\varphi=v\beta-v\beta=0. 
\end{gather*}
Thus, $(t \delta_1-(v\sigma'-n\sigma))\tau'=0, \quad (t \delta_1-(v\sigma'-n\sigma))\tau=0$,
Hence, since $\gamma\sigma=\varphi'\sigma'$ is a pullback, this implies that
$t\delta_1=v\sigma'-n\sigma$. By duality, $\delta_2 s=\tau u-\tau'm$. Consequently,
\begin{multline*}
t\delta_{II}s= t\delta_1\delta_2 s = (v\sigma'-n\sigma)(\tau u-\tau'm) \\
= v\sigma'\tau u - v\sigma'\tau'm - n\sigma\tau u + n\sigma\tau' m 
= v\beta u - v\psi' m - v\beta u = -v\beta u.
\end{multline*} 

The theorem is proved.
\end{proof}

Even having a connecting morphism $\delta:\Ker\gamma\rightarrow\Coker\alpha$, we in general
cannot assert the exactness of the corresponding $\Ker$-$\Coker$-sequence. This exactness
usually requires additional conditions like strictness or semi-stability (see 
\cite{Ge94,GlK02,Grand91-1,KoK00,KoK09}.

\end{document}